\documentclass[a4paper,11pt,reqno]{amsart}

\usepackage[utf8]{inputenc}
\usepackage[left=2.5cm,right=2.5cm,top=3cm, bottom=3cm]{geometry}
\usepackage{amsmath, amssymb, mathrsfs, amsthm}
\usepackage{tikz, tikz-cd}
\usepackage{mathtools}
\usepackage{graphicx}
\usepackage{quiver}
\usepackage[hidelinks]{hyperref}
\usepackage{stmaryrd} 
\usepackage[style=alphabetic,sorting=nyt]{biblatex}
\usepackage{url}
\usepackage{cleveref}

\usepackage{verbatim} 
\usepackage{todonotes}

\def\Spec{\operatorname{Spec}}

\newcommand{\cY}{\mathcal{Y}}
\newcommand{\cX}{\mathcal{X}}

\newtheoremstyle{mytheorem}
{}
{}
{\itshape}
{}
{\scshape}
{:}
{.5em}
{}%

\newtheoremstyle{mydefi}
{}
{}
{}
{}
{\scshape}
{:}
{.5em}
{}%

\theoremstyle{mytheorem}
\newtheorem{teo}{Theorem}
\numberwithin{teo}{section}

\theoremstyle{mydefi}
\newtheorem{ex/}[teo]{Example}

\newtheorem{rmk/}[teo]{Remark}
\newenvironment{rmk}
  {%
   \pushQED{\qed}\begin{rmk/}}
  {\popQED\end{rmk/}}

\newtheorem{defi/}[teo]{Definition}
\newenvironment{defi}
  {%
   \pushQED{\qed}\begin{defi/}}
  {\popQED\end{defi/}}

\newtheorem{ass/}[teo]{Assumption}

\makeatletter
\renewenvironment{proof}[1][\proofname]{\par
  \vspace{-\topsep+6pt}
  \pushQED{\qed}%
  \normalfont
  \topsep0pt \partopsep0pt 
  \trivlist
  \item[\hskip\labelsep
        \scshape
    #1\@addpunct{.}]\ignorespaces
}{%
  \popQED\endtrivlist\@endpefalse
  \addvspace{6pt plus 6pt} 
}
\usepackage{pdfpages}
\usepackage{mathpazo}
\makeatother
\title{Fundamental groups of proper algebraic stacks are finitely presented.}
\date{\today}
\author{Chirantan Chowdhury}
\address{Technische Universit\"at Darmstadt, Fachbereich Mathematik, S215, Schlossgartenstr 7, 64289 Darmstadt}
\email{chowdhury@mathematik.tu-darmstadt.de}

\author{Ludvig Modin}
\address{Leibniz Universität Hannover, Institut für Algebraische Geometrie, Welfengarten 1, 30167 Hannover}
\email{modin@math.uni-hannover.de}

\begin{document}

\begin{abstract}
    In \cite{Propervarietyfinitepresented}, it was proven that the \'etale fundamental group of a proper scheme over an algebraically closed field is topologically finitely presented. 
    Building on this, we show that the \'etale fundamental group, as introduced in \cite{Noohi2002FUNDAMENTALGO}, of a proper algebraic stack with finite inertia over an algebraically closed field is topologically finitely presented. 
\end{abstract}

\maketitle{}

\section{Introduction.}
When the \'etale fundamental group of a scheme was introduced in \cite{SGA1}, it was shown that it is topologically finitely generated if the scheme is proper over an algebraically closed field \cite[Exp X, Theoreme 2.9]{SGA1}. Moreover, in characteristic \(0\) it was shown to be topologically finitely presented for any scheme of finite type over an algebraically closed field. This follows from the \'etale fundamental group of a complex scheme being the profinite completion of the classical fundamental group as is outlined in \cite[Exp. IX, Remarque 5.7]{SGA1}. Over an algebraically closed field of characteristic \(p>0\), \(\pi_1^{et}(\mathbb{A}^1)\) fails to be topologically finitely generated, due to the existence of Artin-Schreier covers. 
The question of topological finite presentation for proper schemes over algebraically closed fields of positive characteristic, raised in \cite[Exp. IX, Remarque 5.7]{SGA1}, was left open until recently. It was resolved for smooth projective curves by Shusterman in \cite{FinPressCurves}, and in \cite[Theorem 1]{FinitepresentationTame}, Esnault, Shusterman and Srinivas showed that the
\'etale fundamental group of a smooth projective variety is topologically finitely presented. Building on this result together with de Jong's alteration theorem \cite[Theorem 4.1]{Alterations-deJong} and descent results developed in \cite{SGA1}, Lara, Srinivas and Stix proved the following.

\begin{teo}\cite[Theorem 1.2]{Propervarietyfinitepresented}\label{theorem: propervarietiesfinitepresentation}
    Let $X$ be a connected scheme that is proper over an algebraically closed field $k$. Then $\pi^{et}_1(X)$ is topologically finitely presented. 
    \end{teo}

The goal of this article is to generalize this result to \'etale fundamental groups of algebraic stacks, as introduced by Noohi in \cite{Noohi2002FUNDAMENTALGO}. 
\begin{teo}
    Let \(\mathcal{X}\) be a connected algebraic stack with finite inertia that is proper over an algebraically closed field $k$. Then \(\pi_1^{et}(\mathcal{X})\) is topologically finitely presented.
\end{teo}

\begin{rmk}
    The assumption of finite inertia on a proper stack is equivalent to the assumption that the stack has affine stabilizers. Indeed, the diagonal of a proper stack is proper, so \(\mathcal{I}_{\mathcal{X}}\to \mathcal{X}\) is proper, and a proper morphism with affine fibers is finite.    
\end{rmk}

The proof uses structural results of \(\pi_1^{et}(\mathcal{X})\) developed in Noohi in \cite{Noohi2002FUNDAMENTALGO} and the Keel-Mori theorem to reduce to proving finite presentation for the \'etale fundamental groups of proper algebraic spaces. We finish the argument by further reducing to the case of proper schemes, using Chow's lemma and descent for finite \'etale covers along surjective proper morphisms of finite type, the same descent results used in \cite{Propervarietyfinitepresented}.

\section*{Acknowledgements}

The authors would like to thank the organizers of the "Summer school on Recent Developments in Arithmetic and Algebraic Geometry in Positive Characteristic" at Bergische Universit\"at Wuppertal, for  providing a nice environment where the article was written. We also want to thank Jochen Heinloth and Dario Wei{\ss}man for helpful comments on a first draft of the article.\\

C. Chowdhury acknowledges support by the European Research Council (ERC) under Horizon Europe (grant agreement nº 101040935), by the Deutsche Forschungsgemeinschaft (DFG, German Research Foundation) TRR 326 \textit{Geometry and Arithmetic of Uniformized Structures}, project number 444845124, and the LOEWE professorship in Algebra, project number LOEWE \\ /4b//519/05/01.002(0004)/87.

L. Modin is currently a postdoctoral researcher at the Institut f\"ur Algebraische Geometrie in Leibniz Universit\"at Hannover.

\section{Preliminaries on \'Etale Fundamental Groups.}
We recall the definition of \'etale fundamental groups of algebraic stacks and some of their important properties due to Noohi(\cite{Noohi2002FUNDAMENTALGO}). 
Let $X$ be a connected scheme over an algebraically closed field $k$. Then for a geometric point $x \in X$, the \'etale fundamental group $\pi_1^{et}(X,x)$ is defined as the automorphism group of the fiber functor ${F_x:\operatorname{F\acute{E}t} \to \mathrm{Sets}}$, taking a finite \'etale covering \(Y\to X\) to the set of points in the fiber of \(x\).

In \cite{Noohi2002FUNDAMENTALGO}, Noohi similarly uses Galois categories and fundamental functors to define fundamental groups for algebraic stacks. The key difference from the case of schemes is that one has to account for points of a stack being isomorphic, but not equal. 
To deal with this Noohi introduced the notion of hidden paths between points of a stack. We recall the key objects and properties. 

\begin{defi}\cite[Definition 3.1]{Noohi2002FUNDAMENTALGO}
Let $x,x': \Spec k \to \cX$ be two geometric points of an algebraic stack \(\mathcal{X}\). A hidden path from $x$ to $x'$, denoted $x \rightsquigarrow x'$, is a natural transformation, i.e a morphism in the category $\operatorname{Fun}(\Spec k ,\cX)$. The hidden fundamental groupoid of $\cX$, denoted by $\Pi_1^h(\cX)$ is defined as follows:
\begin{center}
   $ \operatorname{Ob}(\Pi^h_1(\cX)) = \{x : \Spec k \to \cX\} $\\
    $\operatorname{Mor}(x,x')= \{x \rightsquigarrow x'\} $
\end{center}
    The \textit{hidden fundamental group} of a geometric point $x: \Spec k \to \cX$, denoted by $\pi_1^h(\cX,x)$, is defined as the automorphism group of the object $x \in \Pi_1^h(\cX)$.
\end{defi}
\begin{rmk}
    Note that the hidden fundamental groupoid is the groupoid of \(k\)-points of \(\mathcal{X}\) and the hidden fundamental group of a point \(x\in \mathcal{X}(k)\) is the \(k\)-points of the stabilizer group of \(x\), in particular, if \(\mathcal{X}\) has finite inertia, then the hidden fundamental group at any point is finite. In general \(\pi_1^h(\mathcal{X},x)\) is not profinite, for example if \(\mathcal{X}=B\mathbb{G}_m\), then \(\pi_1^h(\mathcal{X},x)=k^*\). 
\end{rmk}
\begin{defi}\cite[Definition 4.1]{Noohi2002FUNDAMENTALGO}
Let $\cX$ be a connected algebraic stack and let $x: \Spec k \to \cX$ be a geometric point.
\begin{enumerate}
    \item A morphism of connected algebraic stacks $f: \cY  \to \cX$ is a \textit{covering space} if $f$ is finite representable \'etale map.

    \item The Galois category $\mathbf{C}_{\cX}$ and the fundamental functor $F_x$ are defined as follows:
    \begin{enumerate}
        \item  The Galois category $\mathbf{C}_{\cX}$ is the  associated $1$-category of the $(2,1)$- category of covering spaces over $\cX$ by declaring the $2$-morphisms to be identity.
        \item The functor \begin{equation}
            F_x : \mathbf{C}_{\cX} \to \operatorname{Sets}
        \end{equation}
        is defined by 
        \begin{equation}
            F_x(\cY) = \{(y,\phi) \,|\, y :\Spec k \to \cY \, ,\, \phi:  x \rightsquigarrow f(y) \}/ \sim
        \end{equation}
        where  $\sim$ is defined  by 
        \begin{equation}
            (y,\phi) \sim (y',\phi') \,\, \text{if}\, \exists \beta : y \rightsquigarrow y' \, \text{such that} \, f(\beta) \circ \phi = \phi'.
        \end{equation}
         \end{enumerate}
        \item  In a similar way as $\Pi^h_1(\cX)$, one defines the \textit{fundamental groupoid}, denoted by \(\Pi^1(\mathcal{X})\), of $\cX$ as the groupoid with objects being geometric points of $\cX$ and morphisms being natural transformations $F_x \to F_{x'}$. The \textit{fundamental group} of the algebraic stack $\cX$ at a geometric point $x : \Spec k \to \cX$, denoted by $\pi_1^{et}(\cX,x)$, is the automorphism group of the object $x$ in $\Pi^1(\cX)$.
        \item Given two points \(x,x'\), a morphism from \(x'\) to \(x\) in \(\Pi_1(\mathcal{X})\) is called a \emph{path} from \(x'\) to \(x\), denoted by \(x'\to x\).
    
    \end{enumerate}
\end{defi}

\begin{rmk}
\begin{enumerate}
\item Similar to the case of schemes, $\pi_1^{et}(\cX,x)$ is a profinite group and if \(\mathcal{X}\) is connected then \(\Pi_1(\mathcal{X})\) is a connected groupoid, that is, for every pair of points \(x,x'\in \mathcal{X}\) there is a path \(x'\to x\). Any such path induces an isomorphism $\pi_1^{et}(\cX,x')\cong \pi_1^{et}(\cX,x)$.
\item There is a natural morphism
\begin{equation}
    \Pi^h_1(\cX) \to \Pi_1(\cX)
\end{equation}
given by identity on the level of objects and on the level morphisms, a hidden path $\gamma: x \rightsquigarrow x'$ is mapped to the natural transformation $F_x \to F_{x'}$, which, when evaluated at covering $\cY\to \mathcal{X} $ is the function
\begin{equation}
    F_x(\cY) \to F_{x'}(\cY) \,\, (y,\phi) \mapsto (y,\gamma^{-1}\phi).
 \end{equation} 
 In particular, to every hidden path there is an associated (potentially trivial) path.
 This induces a homomorphism (of abstract groups) on the fundamental groups 
 \begin{equation}
     \omega_x : \pi^h_1(\cX,x) \to \pi_1^{et}(\cX,x).
 \end{equation}
 \item For $\cX = BG$ over $k = k^{\operatorname{sep}}$ and $G$ an affine algebraic group of finite type over $k$, Noohi (\cite[Example 4.2]{Noohi2002FUNDAMENTALGO}) shows that $\pi_1^{et}(\cX,x)= G/G^0$, $\pi^h_1(\cX,x)= G(k)$ and the map $\omega_x$ is the canonical quotient map.
  \end{enumerate}  
\end{rmk}

In the case of an algebraic stack $\cX$ admitting a coarse moduli space $X_{\operatorname{mod}}$, one can describe $\pi^{et}_1(X,x)$ in terms of $\pi^{et}_1(\cX,x)$ and the hidden paths of \(\mathcal{X}\). For $x,x' : \Spec k \to \cX$, define 
\begin{equation}
    B_{x'}:= \cup_{\text{paths}\, x' \to x} \operatorname{im}(\pi^h_1(\cX,x') \xrightarrow{\omega_x} \pi^{et}_1(\cX,x') \xrightarrow{\cong}\pi_1^{et}(\cX,x) )
\end{equation}
and let \(N\) be the closure of the subgroup generated by all \(B_x'\)'s as \(x'\) ranges over \(\mathcal{X}(k)\).
This is a normal subgroup of $\pi_1^{et}(\cX,x)$ that maps to zero in $\pi_1^{et}(X_{\operatorname{mod}},x)$. It is in fact the kernel of this map.
\begin{teo}\cite[Theorem 7.11]{Noohi2002FUNDAMENTALGO}\label{theorem: relation to cms}
    The natural map $\pi_1^{et}(\cX,x)/N \to \pi_1^{et}(X_{\operatorname{mod}},x)$ is an isomorphism.
\end{teo}
In general, the kernel $N$ is not trivial, but the following result, which is crucial to our proof of \Cref{Theorem: finite presentation for alg stacks}, shows that \(\mathcal{X}\) admits a cover where it is.
\begin{teo}\cite[Theorem 11.4]{Noohi2002FUNDAMENTALGO}\label{Theorem: existence of covering space}
    Let $\cX$ be a Noetherian algebraic stack. Then there exists a covering space $\cY \to \cX$ such that the image of $\omega_y:\pi_1^h(\mathcal{Y},y)\to\pi_1^{et}(\mathcal{Y},y)$ is the identity for all geometric points of $\cY$. 
\end{teo}

\begin{rmk}
    The proof of the above theorem uses the fact that a Noetherian algebraic stack admits a stratification by monotonous gerbes \cite[Definition 10.1 and Theorem 11.3]{Noohi2002FUNDAMENTALGO} and \cite[Corollary 10.7]{Noohi2002FUNDAMENTALGO}, which controls the image of the hidden fundamental groups of locally closed substacks that are monotonous gerbes. If we impose the additional conditions that $\omega_x$ is injective for all geometric points and $\cX$ is Deligne-Mumford, then $\cY$ is an algebraic space and $\cX = [\cY/G]$ for a finite group $G$ (in other words, $\cX$ is uniformizable) \cite[Remark after Theorem 11.4]{Noohi2002FUNDAMENTALGO}.
\end{rmk}

\section{Results on Finite Presentation.}

In this section, we prove the main theorem on topological finite presentation for proper algebraic stacks. Henceforth, finite presentation will always refer to topological finite presentation. 

We first prove the result for algebraic spaces, which we use to prove the main theorem.
\begin{teo}\label{Theorem: finite presentation for proper alg sp}
    Let $X$ be a connected proper algebraic space over an algebraically closed field $k$. Then $\pi_1^{et}(X,x)$ is finitely presented for all geometric points $x \in X$.
\end{teo}
\begin{proof}
    Let \(X'\to X\) be a proper surjective morphism from a proper scheme (which exists by Chow's lemma for algebraic spaces, \cite[tag 088P]{stacks-project}). As the diagonal of \(X\) is  proper and representable by schemes, it follows that \(X''=X'\times_XX'\) and \(X'''=X'\times_X X'\times_X\times X'\) are proper schemes. By \Cref{theorem: propervarietiesfinitepresentation}, the fundamental groups of the connected components of \(X', X''\) and \(X'''\) are all finitely presented, moreover each of their \(\pi_0\)'s are finite. 
    By \cite[Theoreme 4.12]{SGA1}, the category of finite \'etale covers satisfies descent for proper surjective morphisms of finite type, and by \cite[Exp. IX, Corollarie 5.3]{SGA1} the finite presentation of the fundamental groups of the components of \(X',X'',X'''\) implies that fundamental group of \(X\) also is finitely presented.
\end{proof}

\begin{teo}\label{Theorem: finite presentation for alg stacks}
    Let $\mathcal{X}$ be a connected proper algebraic stack with finite inertia over an algebraically closed field $k$. Then  for any geometric point $x \in \mathcal{X}$, the \'etale fundamental group $\pi_1^{et}(\mathcal{X},x)$ is finitely presented.
\end{teo}
\begin{proof}
Since \(\mathcal{X}\) is proper over a field, it is in particular Noetherian.
By \Cref{Theorem: existence of covering space}, there exists a covering space $\cY \to \cX$ such that the images of $\omega_y:\pi_1^h(\mathcal{Y},y)\to \pi^{et}_1(\mathcal{Y},y)$ are trivial for all geometric points $y \in \cY$. As $\cY \to \cX$ is finite \'etale and $\cX$ is proper with finite inertia, $\cY$ is proper and has finite inertia. By the Keel-Mori theorem(\cite{keelmoritheorem}),  $\cY$ admits a connected coarse moduli space $p: \cY \to Y$ such that $Y$ is proper.  \\

By \Cref{Theorem: finite presentation for proper alg sp}, $\pi_1^{et}(Y,y)$ is finitely presented for all \(y\in Y(k)\).  As $\omega_y$ is trivial for all \(y\in \mathcal{Y}(k)\), the subgroup $N\subset\pi_1^{et}(\mathcal{Y},y)$ is zero. By \Cref{theorem: relation to cms} we see that $\pi_1^{et}(\cY,y) = \pi^{et}_1(Y,y)$ for all geometric points $y$. Thus, $\pi_1^{et}(\cY,y)$ is finitely presented. As $\cY \to \cX$ is a covering space, we have an injection of fundamental groups
\begin{equation}
    \pi_1^{et}(\cY,y) \hookrightarrow \pi_1^{et}(\cX,x)
\end{equation}

where $\pi_1^{et}(\cY,y)$ is an open subgroup of $\pi_1^{et}(\cX,x)$. Hence by \cite[Proposition 2.3]{obstructionlifttochar0}, $\pi_1^{et}(\cX,x)$ is finitely presented.    
\end{proof}

\printbibliography
\end{document}